\documentclass[11pt,oneside]{amsart}
\usepackage[english,francais]{babel}
\usepackage[T1]{fontenc}
\usepackage{amsmath,amsfonts}
\usepackage{amsthm}
\usepackage{pifont}
\usepackage[all]{xy}
\usepackage{tikz}
\usepackage{extarrows}

\usepackage{geometry}

\usetikzlibrary{fit,calc,positioning,decorations.pathreplacing,matrix}
\newtheorem{theorem}{Theorem}

\newtheorem{corollary}[theorem]{Corollary}

\newtheorem{definition}[theorem]{Definition}
\newtheorem{example}[theorem]{Example}

\newtheorem{proposition}[theorem]{Proposition}
\newtheorem{remark}[theorem]{Remark}

\author{A. Boudjaj*} \author{ Y. Rami}
\address{D\'epartement de Math\'ematiques \& Informatique,\\
Universit\'e  My Ismail, B. P. 11 201 Zitoune, Mekn\`es, Morocco,}
\email{azze89ddine@gmail.com  and y.rami@fs-umi.ac.ma}

\title{On spaces of topological complexity two}

\date{\today}


\keywords{Topological complexity, Sectional category,  cup length.\\
{*.This work is a part of the thesis being prepared by the first author.}}
\begin{document}
\maketitle


\selectlanguage{english}
\begin{abstract}
In this paper we consider the classification of minimal cellular structures of  spaces of topological complexity two under some hypotheses on  there graded cohomological algebra. This continues the method used by M.Grant et al. in [1].
\end{abstract}
\section{Introduction \label{intro}}
The {\it topological complexity} (TC) of a space,  introduced by M. Farber, is a numerical homotopy invariant  analogous to the 
 Lusternik-Schnirelmann category (cat) of that space. The two are expressed in terms of sectional category of a specific fibration (see \S \ref{prely}). For instance, $cat(\mathbb{S}^n)$   equals to one either $n$ is odd or even, but $TC(\mathbb{S}^n)$ equals one if $n$ is odd and two if it is even. The goal of this paper is to investigate spaces with TC equals two, continuing the process begun by M. Grant et al. in  [1].\\

Denote by $X$  a non acyclic path-connected and finite type CW-complex (see below). According to Propositions 2.1 and Theorem 3.4.(A) in [1], and by the theorems of Hurewicz Whitehead :\\
  If $X$ is simply-connected and  $dim\widetilde{H}^*(X,\mathbb{K}) =  1$ for all choices of field $\mathbb{K}$, then $TC(X) = 2$ if and only if $X$ is of homotopy type of a sphere with even dimension. \\
  
In the sequel we will assume in addition that $X$ is a non integral homology sphere. 
If $X$is simply-connected and not integral homology sphere. Denote $H_r(X)$, a non zero homology group of $X$, which occurs for an integer $r\geq 1$. Hypothesis implies that $H_r(X)\cong \mathbb{Z}^n \oplus T(H_r(X))$ ($T(H_r(X))$, being the torsion subgroup of $H_r(X)$), with the condition: if $n=1$, we  must have  $T(H_r(X))\not = 0$ or else, there exist another integer $s>r$ such that $H_s(X)\not = 0$. By the universal coefficient theorem, $H^r(X, \mathbb{Z})\cong Hom(H_r(X),\mathbb{Z})\not = 0$. We still use the identification $H_r(X) = \mathbb{Z}^n \oplus T(H_r(X))$.\\

 We call an element  $u$ a \textit{$t$-specific generator}   if it is  the unique generator contained in the free part of  $H^t(X, \mathbb{Z})$, the non zero cohomology group with coefficients in $\mathbb{Z}$.\\
 
 Our first result states that the condition $TC(X) = 2$ influences algebra's structure of $H^*(X,\mathbb{K})$,   for any field $\mathbb{K}$.  Indeed, suppose $a\in H^r(X, \mathbb{Z})$ is a  $r$-specific generator. Two cases are in the list:
\begin{enumerate}
\item[i)] Suppose that $a^2\not = 0$, then, for any field $\mathbb{K}$, $\dim \tilde{H}^*(X, \mathbb{K}) = 2$.
\item[ii)] Suppose that  there is an integer $s > r$  and a $s$-specific generator $b\in H^s(X, \mathbb{Z})$  such that the cup product   $ ab \neq 0 $ , then,   for any field $\mathbb{K}$,  $\dim\tilde{H}^*(X,\mathbb{K}) = 3$.

\end{enumerate}

In the second step of our investigation of spaces having  Topological Complexity equal  two, we focus on the determination of the  cells contained in CW-structure of $X$ for each one of the two cases. To this end, we explicit the homology structure $H_*(X, \mathbb{Z})$.
 assuming (for any field $\mathbb{K}$) first, $dim\tilde{H}^*(X, \mathbb{K}) = 2$ (cf. Theorem 1) and second, $dim\tilde{H}^*(X,\mathbb{K}) = 3$ (cf. Theorem 2).

\begin{theorem}{(corollaries 7, 8)}.
Suppose that $TC(X)  = 2$ and  there exist a $r$-specific generator $a \in H^r(X, \mathbb{Z})$ such that $a^2 \neq 0$  then $\dim \widetilde{H}^* (X, \mathbb{K}) = 2$ for all choices of field  $\mathbb{K}$.
  In addition, $X$ is of  homotopy type of a space which contains only cells of dimensions $r$ and   $2r$.

\end{theorem}

In the second case
\begin{theorem}{(corollaries 9, 11)}.
  Suppose that $TC(X) = 2$ and there exist a $r$-specific generator $a \in H^r(X, \mathbb{Z})$, and a $s$-specific generator $b \in H^s(X, \mathbb{Z})$ 
  such that  
$a b \neq 0$;  then  $\dim \widetilde{H}^* (X, \mathbb{K}) = 3$ for all choices of field $\mathbb{K}$. 
In addition,  $X$ is of homotopy type of 
a CW-complex which contain  only cells of dimensions  $r$, $s$ and  $(r+s)$.
\end{theorem}

\section{Definition and Preliminary results \label{prely}}

The sectional category of a fibration $p : E \longrightarrow B$, denoted by $secat(p)$, is
the smallest number $n$ for which there is an open covering $\{U_0, . . . , U_n \}$ of $B$ by
 $(n + 1)$ open sets, for each of which there is a local section $s_i : U_i \longrightarrow E$ of $p$, so that $p \circ s_i = j_i : U_i \longrightarrow B$, where $j_i$ denotes the inclusion.\\
 Let $X^I$, where $I = [0, 1]$ denote the space of (free) paths on a  space $X$. there is a fibration, substitute from the diagonal map $\Delta : X \longrightarrow X\times X$, $\pi : X^I \longrightarrow X \times X$, which evaluates a continuous path at initial and final point, i.e. for $\alpha \in X^I$, we have $\pi (\alpha) = (\alpha(0), \alpha(1))$.
 
 We define the topological complexity (normalized)  of  $X$, noted  $TC(X)$, to be the sectional category $secat(\pi)$
 of this fibration. That is, $TC(X)$ is the smallest number $n$ for which there is an open cover $\{ U_0, . . ., U_n \}$ of $X\times X$ by $(n+ 1)$ open sets, for each of which there is a local section $s_i : U_i \longrightarrow X^I$  of $\pi$, i.e., for which $\pi \circ s_i = j_i
  : U_i \longrightarrow X\times X$, where $j_i
 $ denotes the inclusion. For instance, 
 $TC(S^n)$ equal 1 for $n$ odd and 2 for n even, e.g. $TC(T^1) = 2$ where $T^2 = S^1 \times S^1$, is the tore.\\

 

 Recall also that  $X$ is of finite type if $H_r(X)$ is a finite integral homology group for any integer $r$. 

 

\begin{definition}
Let $\mathbb{K}$ be a field. the homomorphism induced on cohomology with coefficients in $\mathbb{K}$ by the diagonal map $\Delta : X \longrightarrow X\times  X$ (and thus by  $P_2 : PX  \longrightarrow X\times X$  which is a  fibration substitute for it) may be identified with the \textit{cup} product homomorphism 
$$ U_2(X) : H^* (X, \mathbb{K}) \otimes H^* (X, \mathbb{K}) \longrightarrow H^* (X, \mathbb{K}) $$
The ideal of zero divisors is $\ker U_2(X)$,  the kernel of $U_2(X)$, the \textit{zero-divisors cup length} is  $nil(\ker U_2(X))$, the nilpotency of this ideal. which is to
say the number of factors in the longest non-trivial product of elements from this
ideal.
\end{definition}

We  denote $U_2(X)(a \otimes b)$ and $a\smile b$ by $ab$ (taking into consideration the system  of coefficients).

\begin{proposition}{\label{nilker}}([4, Thm 7])
For any field $\mathbb{K}$,  we have  $$nil(\ker U_2(X)) \leq TC(X).$$
\end{proposition}

The cup product can be defined for more general coefficient groups. For example, the \textit{Alexander-Whitney} diagonal approximation gives a cup product $$ \smile : H^p(X; G_1) \otimes H^q(X; G_2) \longrightarrow H^{p + q}(X; G_1 \otimes G_2),$$
by putting $(a \smile b)(\sigma) = a(\sigma/[v_0, \dots, v_p])\otimes b(\sigma/[v_p, \dots, v_{p+q}]) $. In particular, since  $\mathbb{Z} \otimes \mathbb{Z}/k\mathbb{Z} \cong \mathbb{Z}/k\mathbb{Z}$, there is the \textit{cup product} $H^p(X; \mathbb{Z}) \otimes H^q(X; \mathbb{Z}/k\mathbb{Z}) \longrightarrow H^{p + q}(X; \mathbb{Z}/k\mathbb{Z})$ and it coincides with the product obtained by first reducing mod $k$ and then taking the cup 
product over the ring $\mathbb{Z}/k\mathbb{Z}$. This also the case for any field $K$.
We will denote it by: $$ \smile_1 : H^p(X; \mathbb{Z}) \otimes H^q(X; K) \longrightarrow H^{p + q}(X; K)$$

\section{Proofs of our results \label{proofs}}


Let $X$ be a path connected CW-complex  not acyclic and nor homology sphere. Let $r$, denotes, the smallest integer such that $H^r(X,\mathbb{Z})$ is not trivial. \\

\begin{theorem}\label{thm fon}
Let $n$ be an integer such that $n \geq 2$ so :
\begin{enumerate}

\item If $TC(X) = n$ and   there exist a $r$-specific generator $a \in H^r(X, \mathbb{Z})$ such that  $a^n \neq 0$,   then, $\dim \widetilde{H}^* (X, \mathbb{K}) = n$ for all choices of field  $\mathbb{K}$. Furthermore, $X$  is of homotopy type of a CW-complex  which contains only cells of dimension $r$, $2r$, \dots and $nr$.
\item If  $TC(X) = n$ and suppose that  there exist  $\{a_i / 1\leq i\leq n\}$  a finite set of $r_i$-specific generators (with $r_1 < r_2 < \dots < r_n$)
 such that     $a_1 a_2 \dots a_n \neq 0$,  then  $\dim \widetilde{H}^* (X, \mathbb{K}) = n+1$ for all choices of field $\mathbb{K}$. Furthermore,  $X$  is of homotopy type of a CW-complex which  contains only cells of dimensions  $r_1$, $r_2$, \dots, $r_n$ and $r_1 + r_2 + \dots + r_n$.
 
\end{enumerate}

\end{theorem}
\begin{proof}
Let $n \geq 2$. 
\begin{enumerate}
\item Let $\mathbb{K}$ be a field and  $1_{\mathbb{K}} \in H^0(X, \mathbb{K})$  be the unit element, we have from the hypothesis $a, a^2, \dots, a^n$ are nonzero so $a \smile_1 1_{\mathbb{K}}, a^2 \smile_1 1_{\mathbb{K}}, \dots, a^n \smile_1 1_{\mathbb{K}}$ are nonzero, hence $\dim \widetilde{H}^* (X, \mathbb{K}) \geq n$. \\
Suppose now that there exist $b \in H^s(X, \mathbb{K})$ such that  $a \smile_1 1_{\mathbb{K}}, a^2 \smile_1 1_{\mathbb{K}}, \dots, a^n \smile_1 1_{\mathbb{K}}$ and $b$ are linearly independent over $\mathbb{K}$ so we can proof by induction that  $((a \smile_1 1_{\mathbb{K}}) \otimes 1_{\mathbb{K}} - 1_{\mathbb{K}} \otimes (a \smile_1 1_{\mathbb{K}}))^n (b \otimes 1_{\mathbb{K}} - 1_{\mathbb{K}} \otimes b)$ contains  $ (a^n \smile_1 1_{\mathbb{K}}) \otimes b \pm b \otimes  (a^n \smile_1 1_{\mathbb{K}}) $ which is nonzero because  $a^n \smile_1 1_{\mathbb{K}}$ and $b$ are linearly independent; and since $\mid a \mid \neq \mid b \mid $ hence $((a \smile_1 1_{\mathbb{K}}) \otimes 1_{\mathbb{K}} - 1_{\mathbb{K}} \otimes (a \smile_1 1_{\mathbb{K}}))^n (b \otimes 1_{\mathbb{K}} - 1_{\mathbb{K}} \otimes b)$ is nonzero, then from {\textit{proposition} \ref{nilker}} we have $TC(X) \geq nilker(U_2(X)) \geq n+1$ which is a contradiction.
Therefore $\dim \widetilde{H}^* (X, \mathbb{K}) = n$.\\ 
Suppose  that $\exists s \notin \{r, 2r, \dots, nr \}$ such that $H_s(X) \neq 0$ (where $s \geq 1 $).\\
 if $Free(H_s(x)) \supseteq \mathbb{Z} $ we have for $\mathbb{K} = \mathbb{Q}$ and  from the Universal Coefficient Theorem  (\textit{UCT})   $H^s(X, \mathbb{Q}) \supseteq \mathbb{Q} $ since $H^r(X, \mathbb{Q}) \neq 0$,  $H^{2r}(X, \mathbb{Q}) \neq 0$, \dots, $H^{nr}(X, \mathbb{Q}) \neq 0$  hence $\dim \widetilde{H}^* (X, \mathbb{Q}) \geq n+1$   which is a contradiction.
  else if $\mathbb{Z}/p^k\mathbb{Z} \subset T(H_s(X)) \neq 0$, where $p$ is a prime number, for $K = \mathbb{Z}/p\mathbb{Z} $ we have from the \textit{UCT}  $H^{s+1}(X, \mathbb{Z}/p\mathbb{Z}) \supseteq Ext(H_{s}(X), \mathbb{Z}/p\mathbb{Z}) \supseteq S \neq 0 $ (where $S$ is a $\mathbb{Z}/p\mathbb{Z}$-vector space ) since $H^r(X, \mathbb{Z}/p\mathbb{Z}) \neq 0$,  $H^{2r}(X, \mathbb{Z}/p\mathbb{Z}) \neq 0$, \dots,  $H^{nr}(X, \mathbb{Z}/p\mathbb{Z}) \neq 0$ hence    $\dim \widetilde{H}^* (X, \mathbb{Z}/p\mathbb{Z}) \geq n+1$ which is a contradiction.\\
  Therefor $H_s(X) = 0$ for all $s \notin \{r, 2r, \dots, nr \}$.\\
  Suppose that  $free(H_r(X))  = 0$ so from the \textit{UCT} we have $0 \neq H^r(X,\mathbb{Z}) \cong Free(H_r(X))  \oplus T(H_{r-1}(X))$ then  $\widetilde{H}_{r-1}(X) \neq 0$ which is a contradiction from the previous result.\\
  Suppose now that $H_{t}(X) = \mathbb{Z}^m \oplus T(H_{t}(X))$ for $t \in \{r, 2r, \dots, nr \}$ (where $t \geq 1$). From the \textit{UCT} we have $H^{t}(X,\mathbb{Q}) \supseteq \mathbb{Q}^m$, if $m \geq 2$ then $\dim \widetilde{H}^* (X, \mathbb{Q}) \geq n+1$ which is a contradiction. Hence   $H_{t}(X) = \mathbb{Z}^m \oplus T(H_{t}(X))$ where $m \in \{0, 1\}$. 
     Also if  $\mathbb{Z}/p^k\mathbb{Z} \subset  T(H_{t}(X)) $ so by the \textit{UCT} for $\mathbb{K} = \mathbb{Z}/p\mathbb{Z}$, we have  $H^{t+1}(X, \mathbb{Z}/p\mathbb{Z}) \supseteq \mathbb{Z}/p\mathbb{Z} \oplus T$ (where $T$ is a $\mathbb{Z}/p\mathbb{Z}$-vector space) then $\dim \widetilde{H}^* (X, \mathbb{K}) \geq n+1$ which is a contradiction. Finally $H_{t}(X) = \mathbb{Z}$ for all $t \in \{r, 2r, \dots, nr \} $.\\ 
     Then $X$ is of homotopy type of a CW-complex which  contains only cells of dimensions $r$, $2r$, \dots, and $nr$.

\item Let $\mathbb{K}$ be a field and  $1_{\mathbb{\mathbb{K}}} \in H^0(X, K)$  be the unit element, Clearly, $\dim \widetilde{H}^* (X, \mathbb{K}) \geq n+1$. Suppose now that there exist $b \in H^s(X, \mathbb{K})$ such that $a_1 \smile_1 1_{\mathbb{K}}, a_2 \smile_1 1_{\mathbb{K}}, \dots, a_n \smile_1 1_{\mathbb{K}}, (a_1 \smile_1 1_{\mathbb{K}}) (a_2 \smile_1 1_{\mathbb{K}}) \dots (a_n \smile_1 1_{\mathbb{K}})$ and $b$ are linearly independent over $K$. We can proof by induction that  the term 
$((a_1 \smile_1 1_{\mathbb{K}}) \otimes 1_{\mathbb{K}} - 1_{\mathbb{K}} \otimes (a_1 \smile_1 1_{\mathbb{K}}))((a_2 \smile_1 1_{\mathbb{K}}) \otimes 1_{\mathbb{K}} - 1_{\mathbb{K}} \otimes (a_2 \smile_1 1_{\mathbb{K}})) \dots ((a_n \smile_1 1_{\mathbb{K}}) \otimes 1_{\mathbb{K}} - 1_{\mathbb{K}} \otimes (a_n \smile_1 1_{\mathbb{K}}))(b \otimes 1_{\mathbb{K}} - 1_{\mathbb{K}} \otimes b)$ contains $(a_1 \smile_1 1_{\mathbb{K}}) (a_2 \smile_1 1_{\mathbb{K}}) \dots (a_n \smile_1 1_{\mathbb{K}}) \otimes b$ $ \pm$ $ b \otimes (a_1 \smile_1 1_{\mathbb{K}}) (a_2 \smile_1 1_{\mathbb{K}}) \dots (a_n \smile_1 1_{\mathbb{K}})$ which is nonzero because   $(a_1 \smile_1 1_{\mathbb{K}}) (a_2 \smile_1 1_{\mathbb{K}}) \dots (a_n \smile_1 1_{\mathbb{K}})$ and $b$ are linearly independents over $K$; since $\mid b \mid \not= \mid a_i \mid$ $ \forall i \in I $ so $((a_1 \smile_1 1_{\mathbb{K}}) \otimes 1_{\mathbb{K}} - 1_{\mathbb{K}} \otimes (a_1 \smile_1 1_{\mathbb{K}}))((a_2 \smile_1 1_{\mathbb{K}}) \otimes 1_{\mathbb{K}} - 1_{\mathbb{K}} \otimes (a_2 \smile_1 1_{\mathbb{K}})) \dots ((a_n \smile_1 1_{\mathbb{K}}) \otimes 1_{\mathbb{K}} - 1_{\mathbb{K}} \otimes (a_n \smile_1 1_{\mathbb{K}}))(b \otimes 1_{\mathbb{K}} - 1_{\mathbb{K}} \otimes b)$ is nonzero. Then from \textit{proposition} \ref{nilker} we have $TC(X) \geq nilker(U_2(X)) \geq n+2$ which is a contradiction.\\
Therefor   $\dim \widetilde{H}^* (X, \mathbb{K}) = n+1$.\\


Suppose now that $\exists s \notin \{r_1, r_2, \dots,  r_n, r_1 + \dots + r_n\}$ (where $s \geq 1 $) such that $H_s(X) \neq 0$. \\
If at least $H_s(X) \supseteq \mathbb{Z}$ so we have from the \textit{UCT} $H^s(X, \mathbb{Q}) \supseteq  \mathbb{Q} $, since $H^{r_i}(X, \mathbb{Q}) \neq 0 $ for $i \in I$ and $H^{r_1 + \dots + r_n}(X, \mathbb{Q}) \neq 0$  hence $\dim \widetilde{H}^* (X, \mathbb{Q}) \geq n+2 $
 which is a contradiction. \\
 Also if $T(H_s(X)) \neq 0$ so  there exist at least a non zero summand $\mathbb{Z}/p^{k}\mathbb{Z}$, where $p$ is a prime number, so from the \textit{UCT} we have $H^{s}(X, \mathbb{Z}/p\mathbb{Z}) \supseteq \mathbb{Z}/p\mathbb{Z} \oplus S  $ and $H^{s+1}(X, \mathbb{Z}/p\mathbb{Z}) \supseteq \mathbb{Z}/p\mathbb{Z} \oplus T  $ (where $S$ and $T$ are two $\mathbb{Z}/p\mathbb{Z}$-vector spaces), since   $H^{r_i}(X, \mathbb{Z}/p\mathbb{Z}) \neq 0$  for $i \in I$ and $H^{r_1 + \dots + r_n}(X, \mathbb{Z}/p\mathbb{Z}) \neq 0$  then $\dim \widetilde{H}^* (X, \mathbb{Z}/\mathbb{Z}) \geq n+3 $ which is also  a contradiction.\\
 So for all integer $s \notin \{r_1, \dots, r_n, r_1 + \dots + r_n\}$, $H_s(X)$ is trivial.\\

 As in the proof of the first part in the theorem and from the result above  we have $Free(H_r(X)) \neq 0$. \\
 Suppose that $ \exists t \in \{ r_1, \dots, r_n, r_1 + \dots + r_n \} $ such that  $H_t(X) \cong \mathbb{Z}^n \oplus T(H_t(X))$ (where $t \geq 1 $). If $T(H_t(X))$  has at least a summand $\mathbb{Z}/p^k\mathbb{Z}$ and $n \geq 1$ so by the  \textit{UCT} we have  $H^t(X, \mathbb{Z}/p\mathbb{Z}) \supseteq \mathbb{Z}/p\mathbb{Z} \oplus \mathbb{Z}/p\mathbb{Z} \oplus S$, where $S$ is a $\mathbb{Z}/p\mathbb{Z}$-vector space, hence $\dim \widetilde{H}^* (X, \mathbb{Z}/p\mathbb{Z}) \geq n+2$ which is a contradiction  then $n = 0$ or $T(H_t(X)) = 0$. 
 Suppose now  that  $T(H_t(X))$  has at least a summand $\mathbb{Z}/p^k\mathbb{Z}$ hence $n=0$, since $H^t(X,\mathbb{Z}) = Free(H_t(X)) \oplus T(H_{t-1}(X))$ so $T(H_{t-1}(X)) \neq 0$ if  $t-1 \notin \{ r_1, \dots, r_n, r_1 + \dots + r_n \}$ ) it's a contradiction  else (i.e.
  $t-1 \in \{ r_1, \dots, r_n, r_1 + \dots + r_n \}$) we proceed in the same manner and we will have $T(H_{t-2}(X)) \neq 0$ we continue the process till to get an element not in $\{ r_1, \dots, r_n, r_1 + \dots + r_n \}$ which is a contradiction. \\
 Therefor $n = 1$ and $T(H_t(X)) = 0$ i.e. $H_t(X) \cong \mathbb{Z}$,  $\forall t \in \{ r_1, \dots, r_n, r_1 + \dots + r_n \} $.\\
 Then   $X$  is of homotopy type of a CW-complex which  contains only cells of dimensions $r_1$, $r_2$, \dots, $r_n$, and $r_1 + r_2 + \dots + r_n$ \\

\end{enumerate}

\end{proof}

\begin{remark}
\begin{enumerate}

\item The results of the previous theorem holds true if   $\mid a_1 \mid =  \dots  = \mid a_n \mid = r$ and $a_1, \dots, a_n$ are the unique generators of degree $r$. 

\item Since the cup product of the wedge of spheres  is  trivial,  the space  $X$ which verify the hypothesis of the  \textit{theorem 5} won't  necessary be wedge of spheres. In our next work, we will focus on determining  homotopy types of attaching maps in minimal structures for various cases cited in the previous theorem 
\end{enumerate}

\end{remark}

 

Applying \textit{theorem} \ref{thm fon} in the case of topological complexity two, we have immediately the :
\begin{corollary}
Suppose that $TC(X)  = 2$ and  there exist a  $r$-specific generator $a \in H^r(X, \mathbb{Z})$ such that $a^2 \neq 0$  hence $\dim \widetilde{H}^* (X, K) = 2$ for all field  $K$.
\end{corollary}

\begin{corollary}
With the same hypothesis as in \textit{corollary 7},   we have 
  $H_r(X) = \mathbb{Z}$, $H_{2r}(X) = \mathbb{Z}$ and $H_s(X) = 0$ for all $s \notin \{  r, 2r \}$,  Furthermore  $X$ is of homotopy type of a CW-complex which  contains only cells of dimensions $r$ and $2r$
 
\end{corollary}

\begin{corollary}
  Suppose that $TC(X) = 2$ and there exist a $r$-specific generator $a \in H^r(X, \mathbb{Z})$, and a $s$-specific generator $b \in H^s(X, \mathbb{Z})$ (where $r < s$) 
   such that  
 $a b \neq 0$;  therefore  $\dim \widetilde{H}^* (X, K) = 3$ for all field $K$.
 \end{corollary}
 The previous corollary holds  true if $ \mid a \mid = \mid b \mid = r $ and are the unique generators of degree $r$. \\
 
 As an example we take the case of the Tore :

\begin{example}
If $X = T^2$. we have $TC(T^2) = 2$ and there exist two generators  $a$ and $b$ which they  generate $ H^1(X,\mathbb{Z}) \cong \mathbb{Z} \oplus \mathbb{Z}$ such that $ab$ generates $H^2(X,\mathbb{Z}) \cong \mathbb{Z}$, so  $\dim \widetilde{H}^* (X, K) = 3$ for all field $K$. 
\end{example}

\begin{corollary}
with the same properties as in  \textit{corollary 9}, we have $X$ is of homotopy type of a CW-complex which  contains only cells of dimensions $r$, $s$ and $r+s$
  
\end{corollary}

\vspace*{1cm}

\textit{Acknowledgment : we are indebted  to members of Moroccan Area of Algebraic Topology group for very useful conversations  which occurred at its  monthly meetings, also we are grateful to Mr.M.Grant for his precious remarks.}

\end{document}